\documentclass[11pt]{amsart}

\usepackage{amsmath,amssymb,amsthm,mathtools}
\usepackage[margin=1in]{geometry}
\usepackage{microtype}
\usepackage{xcolor}
\definecolor{citationgreen}{HTML}{07eb57}
\usepackage[
  colorlinks=false,
  linkbordercolor=red,
  citebordercolor=citationgreen,
  urlbordercolor=cyan,
  pdfborder={0 0 1}
]{hyperref}

\numberwithin{equation}{section}

\newtheorem{thm}{Theorem}[section]
\newtheorem{prop}[thm]{Proposition}
\newtheorem{cor}[thm]{Corollary}
\newtheorem{lem}[thm]{Lemma}

\theoremstyle{definition}
\newtheorem{defn}[thm]{Definition}

\theoremstyle{remark}

\newcommand{\B}{\mathcal{B}}
\newcommand{\cE}{\mathcal{E}}
\newcommand{\ccap}{\operatorname{cap}}
\newcommand{\F}{\mathcal{F}}
\newcommand{\hh}{\mathcal{H}}
\newcommand{\Hinf}{\hh^\infty}
\newcommand{\N}{\mathbb{N}}
\newcommand{\R}{\mathbb{R}}
\newcommand{\supp}{\operatorname{supp}}

\title[Products with a Parabolic Factor]
{Bounded Harmonic Functions on Products with a Parabolic Factor}
\author{Ruotong Jia}
\address{Key Laboratory of Pure and Applied Mathematics,
	School of Mathematical Sciences, Peking University, Beijing, 100871, P. R. China
}
\email{2601110010@stu.pku.edu.cn}

\begin{document}

\begin{abstract}
We prove that if $M$ is a connected complete parabolic Riemannian manifold and
$N$ is a connected complete stochastically complete Riemannian manifold, then
every bounded harmonic function on $M\times N$ is independent of the
$M$-variable. Equivalently, pullback by the second projection induces an
isometric isomorphism from the space of bounded harmonic functions on $N$
onto that on $M\times N$. In particular, the product of two parabolic
manifolds has the bounded Liouville property, thereby answering Problem~16 of
Grigor'yan's survey in the affirmative. The key analytic input is a
total-variation memory-loss property of the heat kernel on a parabolic
manifold. We establish this property by showing that the time-one heat kernel
defines an aperiodic Harris recurrent transition kernel and then applying the
row-merging theorem of Jamison and Orey.
\end{abstract}

\maketitle

\tableofcontents

\section{Introduction}

For a connected complete Riemannian manifold $M$ without boundary, set
$$
\Hinf(M)\coloneq
\{u\in C^\infty(M):\Delta_M u=0,\ \|u\|_{L^\infty}<\infty\}.
$$
Here $\Delta_M$ denotes the Laplace--Beltrami operator. All functions are
real-valued, and $\Hinf(M)$ is equipped with the $L^\infty$ norm. The manifold
$M$ has the \emph{bounded Liouville property} if $\Hinf(M)$ consists only of
constant functions. Every parabolic manifold has this property. Parabolicity
itself is not preserved under products: $\R$ and $\R^2$ are parabolic, whereas
$\R^3=\R\times\R^2$ is not. Nevertheless, the classical Liouville theorem
shows that $\R^3$ still has the bounded Liouville property. Thus parabolicity
is sufficient, but not necessary, for the bounded Liouville property. It is
therefore natural to ask whether the latter property persists under products
of parabolic manifolds. Grigor'yan formulated precisely this question
in~\cite[Problem~16]{Grigorian}.

Let $\pi_N:M\times N\to N$ denote the projection $\pi_N(x,y)=y$. The
pullback of a function $v$ is given by $(\pi_N^*v)(x,y)=v(y)$.

We prove the following asymmetric statement, which is stronger than the
assertion for two parabolic factors.

\begin{thm}\label{thm:main}
Let $M$ and $N$ be connected complete Riemannian manifolds without boundary.
Assume that $M$ is parabolic and that $N$ is stochastically complete. Then
every bounded harmonic function on $M\times N$ is the pullback of a unique
bounded harmonic function on $N$. Consequently, pullback by the second
projection defines an isometric linear isomorphism
$$
\pi_N^*:\Hinf(N)\longrightarrow \Hinf(M\times N).
$$
\end{thm}

Since every parabolic manifold is stochastically complete and has the bounded
Liouville property, taking $N$ parabolic in Theorem~\ref{thm:main} gives the
affirmative conclusion requested in Grigor'yan's problem.

The distinction between parabolicity and the bounded Liouville property is
essential. Grigor'yan observed that, for every nontrivial compact
Riemannian manifold $K$, the product $N\times K$ has the bounded Liouville
property if and only if $N$ has that property and is stochastically
complete~\cite[Section~13.5]{Grigorian}. Pinchover constructed a complete but
stochastically incomplete manifold with the bounded Liouville
property~\cite{Pinchover_1995}.
Consequently, the assumption that both factors merely have the bounded
Liouville property cannot replace parabolicity in the product question.

Martin-boundary theory provides another perspective on the product problem.
Molchanov first studied this boundary for products of Markov chains and then
for products of Markov processes~\cite{Molchanov_1967,Molchanov_1970}. In the
Riemannian setting, Freire proved that if both factors are complete and
noncompact and have Ricci curvature bounded below, then every bounded harmonic
function on their product is separately harmonic in the two
variables~\cite[Corollary, p.~216]{Freire_1991}. Hence Grigor'yan's question
has an affirmative answer under these curvature hypotheses.
Theorem~\ref{thm:main} removes the curvature hypotheses from the
bounded-harmonic problem posed by Grigor'yan.

The key analytic input in our proof is that the heat kernel of a parabolic
manifold asymptotically forgets its initial point in total variation. Let
$p_t^M$ denote the minimal heat kernel of $M$.

\begin{thm}[Memory loss on parabolic manifolds]\label{thm:memory-loss}
Let $M$ be a connected complete parabolic Riemannian manifold of dimension
$m$. Then, for all $x,x'\in M$,
\begin{equation}\label{eq:memory-loss}
\lim_{t\to\infty}
\int_M |p_t^M(x,y)-p_t^M(x',y)|\,d\hh_M^m(y)=0.
\end{equation}
\end{thm}

Its role in the product theorem is as follows. For $u\in\Hinf(M\times N)$,
heat-semigroup invariance, the product heat-kernel formula, and stochastic
completeness of $N$ yield
$$
|u(x,y)-u(x',y)|
\le \|u\|_{L^\infty(M\times N)}
\int_M |p_t^M(x,z)-p_t^M(x',z)|\,d\hh_M^m(z).
$$
By \eqref{eq:memory-loss}, the right-hand side tends to zero; hence $u$ is
independent of the $M$-variable.

Total-variation row merging is classical in the theory of Markov chains.
Orey proved
it for irreducible aperiodic recurrent chains on countable state
spaces~\cite{Orey_1962}, and Jamison and Orey established the corresponding
result for Harris recurrent chains on general measurable
spaces~\cite{Jamison_1967}. Thus Theorem~\ref{thm:memory-loss} is not a new
abstract row-merging theorem. What is needed for the product problem is the
verification that the unit-time heat-kernel operator
$$
Kf(x)=\int_M p_1^M(x,y)f(y)\,d\hh_M^m(y)
$$
falls within that framework. A parabolic cutoff sequence establishes
recurrence of the discrete Dirichlet form associated with $K$. A
reduced-function argument then verifies the Harris recurrence hypothesis,
while strict positivity of the heat kernel eliminates any nontrivial cyclic
decomposition. The Jamison--Orey theorem yields merging at integer times, and
the $L^1$-contraction property of the heat semigroup extends the limit to
arbitrary real times.

Section~\ref{sec:preliminaries} collects the geometric and analytic
preliminaries.
Section~\ref{sec:integral-operators} develops the recurrent-kernel argument and
states the Jamison--Orey theorem in the notation used here. The final section
applies the memory-loss estimate to the product heat kernel and proves
Theorems~\ref{thm:memory-loss} and~\ref{thm:main}.

\section{Preliminaries}\label{sec:preliminaries}

\subsection{Hausdorff measure and the heat semigroup}

Let $M$ be a connected complete Riemannian manifold of dimension $m$
without boundary. Its Riemannian distance induces the $m$-dimensional
Hausdorff measure $\hh_M^m$, normalized to agree with Lebesgue measure on
$\R^m$. All $L^p(M)$ spaces and heat-kernel integrals below are understood
with respect to $\hh_M^m$. Since $M$ is second countable, its Borel
$\sigma$-algebra
$\B(M)$ is countably generated. Moreover, a countable relatively compact
exhaustion shows that $\hh_M^m$ is $\sigma$-finite.

The minimal heat kernel $p_t^M(x,y)$ is the minimal nonnegative fundamental
solution of the heat equation. It exists on every Riemannian manifold. In the
present connected setting, it is smooth and strictly positive for $t>0$. It
is symmetric and satisfies the semigroup identity
\begin{equation}\label{eq:heat-kernel-identities}
\begin{aligned}
p_t^M(x,y)&=p_t^M(y,x),\\
p_{t+s}^M(x,y)
&=\int_M p_t^M(x,z)p_s^M(z,y)\,d\hh_M^m(z).
\end{aligned}
\end{equation}

The associated minimal heat semigroup is
$$
P_t^M f(x)\coloneq\int_M p_t^M(x,y)f(y)\,d\hh_M^m(y).
$$

It is positivity preserving and sub-Markovian. In particular,
$$
0<P_t^M 1(x)=\int_M p_t^M(x,y)\,d\hh_M^m(y)\le1.
$$

\begin{lem}\label{lem:heat-l1-contraction}
Let $M$ be an arbitrary Riemannian manifold without boundary. Then, for every
$t>0$ and every $f\in L^1(M)$,
$$
\|P_t^M f\|_{L^1(M)}\le\|f\|_{L^1(M)}.
$$
\end{lem}

\begin{proof}
Positivity gives $|P_t^M f|\le P_t^M|f|$. By symmetry, Tonelli's theorem,
and the sub-Markov property,
$$
\begin{aligned}
\|P_t^M f\|_{L^1(M)}
&\le\int_M P_t^M |f|(x)\,d\hh_M^m(x)\\
&=\int_M |f|(y)P_t^M 1(y)\,d\hh_M^m(y)\\
&\le\|f\|_{L^1(M)}.
\end{aligned}
$$
\end{proof}

\begin{defn}\label{def:stochastic-complete}
The manifold $M$ is \emph{stochastically complete} if
\begin{equation}\label{eq:stochastic-completeness}
P_t^M 1(x)=\int_M p_t^M(x,y)\,d\hh_M^m(y)=1,
\end{equation}
for every $t>0$ and every $x\in M$.
\end{defn}

Although the terminology is probabilistic,
Definition~\ref{def:stochastic-complete} is simply the analytic statement
that the minimal heat kernel preserves total mass.

\subsection{Parabolicity and cutoff functions}

\begin{defn}\label{def:parabolic}
The manifold $M$ is \emph{parabolic} if
$$
\int_0^\infty p_t^M(x,y)\,dt=\infty,
$$
for all distinct $x,y\in M$. For a connected manifold, it is enough to
require this divergence for a single pair of distinct points.
\end{defn}

A $C^2$ function $v$ is superharmonic if $\Delta_M v\le0$.
Parabolicity is equivalent to every positive superharmonic function being
constant; see~\cite[Theorem~5.1]{Grigorian}.

For a compact set $K\subset M$, its Newtonian capacity is
$$
\ccap_M(K)\coloneq
\inf\left\{
\int_M |d\varphi|^2\,d\hh_M^m:
\varphi\in C_c^\infty(M),
\varphi\ge1\text{ on a neighborhood of }K
\right\}.
$$
The capacity criterion~\cite[Theorem~5.1(6)]{Grigorian} asserts that $M$ is
parabolic if and only if $\ccap_M(K)=0$ for every compact set
$K\subset M$.

\begin{lem}[Parabolic cutoff sequence]\label{lem:parabolic-cutoff}
If $M$ is parabolic, then there exists a sequence
$\{\eta_j\}_{j=1}^\infty\subset C_c^\infty(M)$ such that
$$
0\le\eta_j\le1\qquad\text{on }M,
$$
for every $j$. For every compact set $K\subset M$,
$$
\lim_{j\to\infty}\sup_{x\in K}|\eta_j(x)-1|=0.
$$
Moreover,
\begin{equation}\label{eq:cutoff-energy-decay}
\lim_{j\to\infty}
\int_M |d\eta_j|^2\,d\hh_M^m=0.
\end{equation}
\end{lem}

\begin{proof}
If $M$ is compact, take $\eta_j\equiv1$. Suppose that $M$ is
noncompact, and choose a compact exhaustion
$$
K_j\subset\operatorname{int}K_{j+1},
\qquad
M=\bigcup_{j=1}^\infty\operatorname{int}K_j.
$$
Fix a smooth nondecreasing function $\Theta:\R\to[0,1]$ such that
$\Theta=0$ on $(-\infty,0]$ and $\Theta=1$ on $[1,\infty)$, and
put $C_\Theta\coloneq\|\Theta'\|_\infty$. Since $M$ is parabolic, the capacity
criterion gives $\ccap_M(K_j)=0$. Thus one can choose
$\varphi_j\in C_c^\infty(M)$, with $\varphi_j\ge1$ on a neighborhood of
$K_j$, such that
$$
\int_M |d\varphi_j|^2\,d\hh_M^m<\frac{1}{jC_\Theta^2}.
$$
Set $\eta_j\coloneq\Theta\circ\varphi_j$. Then $\eta_j$ is smooth and
compactly supported, because $\Theta(0)=0$. Moreover,
$$
0\le\eta_j\le1,
\qquad
\eta_j\equiv1\text{ on a neighborhood of }K_j.
$$
The chain rule yields
$$
\int_M |d\eta_j|^2\,d\hh_M^m
\le C_\Theta^2\int_M |d\varphi_j|^2\,d\hh_M^m
<\frac{1}{j}.
$$
Finally, every compact set $K\subset M$ is contained in some $K_{j_0}$.
For every $j\ge j_0$, one has $\eta_j\equiv1$ on $K$. This proves the
asserted local uniform convergence.
\end{proof}

The next estimate relates the continuous heat semigroup to the Dirichlet
energy.

\begin{lem}\label{lem:heat-energy-estimate}
For every $f\in C_c^\infty(M)$ and every $t>0$,
\begin{equation}\label{eq:heat-energy-estimate}
0\le\langle f,(I-P_t^M)f\rangle_{L^2(M)}
\le t\int_M |df|^2\,d\hh_M^m.
\end{equation}
\end{lem}

\begin{proof}
Choose an exhaustion $\{\Omega_j\}_{j=1}^\infty$ of $M$ by smooth,
relatively compact domains such that $\supp f\subset\Omega_1$. Let
$p_t^{\Omega_j}$ and $P_t^{\Omega_j}$ denote the Dirichlet heat kernel and
Dirichlet heat semigroup on $\Omega_j$, respectively, and extend
$p_t^{\Omega_j}$ by zero outside $\Omega_j\times\Omega_j$. The construction
of the minimal heat kernel gives
$$
p_t^{\Omega_j}(x,y)\nearrow p_t^M(x,y),
$$
for every $t>0$ and all $x,y\in M$.

Fix $j$ and set $u(s,\cdot)=P_s^{\Omega_j}f$. Then
$$
\partial_s u=\Delta_M u\quad\text{on }\Omega_j,
\qquad
u=0\quad\text{on }\partial\Omega_j,
\qquad
u(0,\cdot)=f.
$$
Define
$$
\cE_j(s)\coloneq
\int_{\Omega_j}|du(s,\cdot)|^2\,d\hh_M^m.
$$
Green's identity and the heat equation yield
$$
\frac{d}{ds}\|u(s,\cdot)\|_{L^2(\Omega_j)}^2
=2\int_{\Omega_j}u\Delta_M u\,d\hh_M^m
=-2\cE_j(s).
$$
Since $u(s,\cdot)$ vanishes on $\partial\Omega_j$ for every $s$, so does
$\partial_s u$. A second application of Green's identity therefore gives
$$
\begin{aligned}
\cE_j'(s)
&=2\int_{\Omega_j}
  \langle du,d(\partial_s u)\rangle\,d\hh_M^m\\
&=-2\int_{\Omega_j}(\partial_s u)\Delta_M u\,d\hh_M^m\\
&=-2\int_{\Omega_j}|\Delta_M u|^2\,d\hh_M^m\le0.
\end{aligned}
$$
These identities remain valid at $s=0$, since $f$ vanishes near
$\partial\Omega_j$.
Thus $\cE_j(s)\le\cE_j(0)=\int_M|df|^2\,d\hh_M^m$. By symmetry and the
semigroup property of the Dirichlet heat semigroup,
$$
\langle f,P_t^{\Omega_j}f\rangle_{L^2(\Omega_j)}
=\|P_{\frac{t}{2}}^{\Omega_j}f\|_{L^2(\Omega_j)}^2.
$$
Consequently,
\begin{equation}\label{eq:dirichlet-energy-estimate}
\begin{aligned}
0\le{}
\langle f,(I-P_t^{\Omega_j})f\rangle_{L^2(\Omega_j)}
&=\|f\|_{L^2(\Omega_j)}^2
  -\|u(\frac{t}{2},\cdot)\|_{L^2(\Omega_j)}^2\\
&=-\int_0^{\frac{t}{2}}\frac{d}{ds}\|u(s,\cdot)\|_{L^2(\Omega_j)}^2\,ds\\
&=2\int_0^{\frac{t}{2}}\cE_j(s)\,ds\\
&\le t\int_M|df|^2\,d\hh_M^m.
\end{aligned}
\end{equation}

The dominated convergence theorem implies
$$
\lim_{j\to\infty}\langle f,P_t^{\Omega_j}f\rangle_{L^2(\Omega_j)}=
\langle f,P_t^M f\rangle_{L^2(M)}.
$$
Passing to the limit in \eqref{eq:dirichlet-energy-estimate} proves the
assertion.
\end{proof}

We also use the standard fact that parabolicity implies stochastic
completeness; see~\cite[Corollary~6.4]{Grigorian}.

\begin{prop}\label{prop:parabolic-stochastic-complete}
Every connected complete parabolic Riemannian manifold is stochastically
complete.
\end{prop}

\subsection{Bounded heat solutions and Riemannian products}

The following uniqueness theorem is another standard characterization of
stochastic completeness; see~\cite[Theorem~6.2 and
Corollary~6.3]{Grigorian}. That reference uses the generator
$\frac{1}{2}\Delta_M$; rescaling the time variable gives the normalization
used here.

\begin{thm}\label{thm:bounded-heat-uniqueness}
A connected complete Riemannian manifold $M$ is stochastically complete if
and only if the following uniqueness property holds: for every $T>0$, the
only bounded function
$$
W\in C^{2,1}(M\times(0,T])\cap C(M\times[0,T])
$$
satisfying
$$
\partial_t W=\Delta_M W,
\qquad
W(\cdot,0)=0,
$$
is $W\equiv0$.
\end{thm}

\begin{cor}\label{cor:harmonic-fixed-by-heat}
If $M$ is stochastically complete and
$h\in\Hinf(M)$, then
\begin{equation}\label{eq:harmonic-heat-invariance}
P_t^M h=h,
\end{equation}
for every $t>0$.
\end{cor}

\begin{proof}
The functions $U(x,t)=h(x)$ and $V(x,t)=P_t^M h(x)$ are bounded
solutions of the heat equation with the same initial value $h$.
Theorem~\ref{thm:bounded-heat-uniqueness}, applied to $U-V$, gives
$U=V$.
\end{proof}

We record the product identities used below.

\begin{lem}\label{lem:product-heat-kernel}
Let $M$ and $N$ be connected complete Riemannian manifolds of dimensions
$m$ and $n$, respectively. Then
\begin{align}
\hh_{M\times N}^{m+n}
&=\hh_M^m\otimes\hh_N^n,\label{eq:product-measure}\\
\Delta_{M\times N}
&=\Delta_M+\Delta_N,\label{eq:product-laplacian}\\
p_t^{M\times N}\bigl((x,y),(z,w)\bigr)
&=p_t^M(x,z)p_t^N(y,w).\label{eq:product-heat-kernel}
\end{align}
Equation~\eqref{eq:product-heat-kernel} holds for every $t>0$, $x,z\in M$, and
$y,w\in N$.
Consequently, if $M$ and $N$ are stochastically complete, then so is
$M\times N$.
\end{lem}

\section{Recurrent positive integral operators}\label{sec:integral-operators}

This section develops the abstract integral-operator result used in the proof.
All functions in this section are real-valued. Let $(X,\B,\mu)$ be a
countably generated, $\sigma$-finite measure space, and let
$k:X\times X\to(0,\infty)$ be jointly measurable. Assume that
$$
k(x,y)=k(y,x),
\qquad
x,y\in X,
$$
and that
$$
\int_X k(x,y)\,d\mu(y)=1,
\qquad
x\in X.
$$

For each $x\in X$, the measure
$$
\nu_x(A)\coloneq\int_A k(x,y)\,d\mu(y),
\qquad
A\in\B
$$
is a probability measure on $X$.

Let $\B_b(X)$ denote the space of bounded $\B$-measurable
functions on $X$. For $f\in\B_b(X)$, define
$$
Kf(x)=\int_X k(x,y)f(y)\,d\mu(y),
\qquad
x\in X.
$$

The operator $K$ is positive and conservative on $\B_b(X)$, and its kernel is
symmetric. The same integral formula defines $Kf$, with values in
$[0,\infty]$, for every nonnegative $\B$-measurable function $f$. For such
$f$, Tonelli's theorem, symmetry, and conservativity yield the invariance
identity
\begin{equation}\label{eq:measure-invariance}
\int_X Kf\,d\mu=\int_X f\,d\mu,
\end{equation}
where the equality is understood in $[0,\infty]$. In other words, $\mu$ is
an invariant measure.

For $f\in\B_b(X)\cap L^2(X)$, the Cauchy--Schwarz inequality and
\eqref{eq:measure-invariance} imply
$$
\|Kf\|_{L^2}^2\le\int_X K(f^2)\,d\mu=\|f\|_{L^2}^2.
$$

Consequently, $K$ extends uniquely to a contraction on $L^2(X)$, and
symmetry of $k$ makes this extension self-adjoint.

For $f, g\in L^2(X)$, the associated discrete Dirichlet form is defined by
$$
\begin{aligned}
\cE_K(f,g)
&\coloneq\langle f,(I-K)g\rangle_{L^2}\\
&=\frac{1}{2}\iint_{X\times X}
\bigl(f(x)-f(y)\bigr)\bigl(g(x)-g(y)\bigr)
k(x,y)\,d\mu(x)\,d\mu(y).
\end{aligned}
$$

\begin{defn}\label{def:recurrent-form}
The form $\cE_K$ is \emph{recurrent} if there exists a sequence of
measurable functions $\{f_j\}_{j=1}^\infty$, each supported on a set of
finite $\mu$-measure, such that
$$
0\le f_j\le1,
\qquad
\lim_{j\to\infty}f_j=1,
\qquad
\mu\text{-a.e.},
$$
and
$$
\lim_{j\to\infty}\cE_K(f_j,f_j)=0.
$$
\end{defn}

\subsection{Reduced functions}
For $A\in\B$, define
$$
r_{A,0}=1_A,
$$
and, recursively,
$$
r_{A,m+1}=1_A+1_{A^c}Kr_{A,m}.
$$
Here $1_A$ is the indicator function of $A$. We call a nonnegative
measurable function $v$ $K$-\emph{excessive} if $Kv\le v$, and
$K$-\emph{harmonic} if $Kv=v$.

\begin{lem}\label{lem:reduction-properties}
The following statements hold.
\begin{enumerate}
\item $0\le r_{A,m}\le1$ for every $m$.
\item $r_{A,m}\le r_{A,m+1}$ for every $m$.
\item The pointwise limit
$$
r_A\coloneq\lim_{m\to\infty}r_{A,m},
$$
exists and satisfies
$$
r_A=1_A+1_{A^c}Kr_A.
$$
\item In particular,
$$
Kr_A\le r_A.
$$
\item The function $r_A$ is the least nonnegative $K$-excessive majorant of
$1_A$.
\end{enumerate}
\end{lem}

\begin{proof}
We prove the first statement by induction. The base case
$r_{A,0}=1_A$ is immediate. If $0\le r_{A,m}\le1$, then positivity and
$K1=1$ imply $0\le Kr_{A,m}\le1$. By its recursive definition,
$r_{A,m+1}$ equals one on $A$ and $Kr_{A,m}$ on $A^c$. Hence
$0\le r_{A,m+1}\le1$.

We next prove monotonicity. The base case $r_{A,0}\le r_{A,1}$ is
immediate. If $r_{A,m-1}\le r_{A,m}$, then positivity of $K$ implies
$$
1_A+1_{A^c}Kr_{A,m-1}
\le1_A+1_{A^c}Kr_{A,m}.
$$
This completes the induction.

Since the sequence is increasing and bounded, it converges pointwise to
$r_A$. Applying the monotone convergence theorem to the integral defining
$K$ shows that
$$
r_A=1_A+1_{A^c}Kr_A.
$$

On $A^c$, this implies $r_A=Kr_A$. On $A$, one has $r_A=1$ and
$Kr_A\le K1=1$. Hence $Kr_A\le r_A$ everywhere.

Finally, suppose $v\ge0$, $v\ge1_A$, and $Kv\le v$. We prove by induction
that $r_{A,m}\le v$ for every $m$. The base case is
$$
r_{A,0}=1_A\le v,
$$
and if $r_{A,m}\le v$, then
$$
\begin{aligned}
r_{A,m+1}
&=1_A+1_{A^c}Kr_{A,m}\\
&\le1_A+1_{A^c}Kv\le v.
\end{aligned}
$$
Letting $m\to\infty$ proves that $r_A\le v$.
\end{proof}

For $m\ge1$, define
$$
H_{A,m}\coloneq Kr_{A,m-1}.
$$

By monotone convergence,
$$
H_A\coloneq\lim_{m\to\infty}H_{A,m}=Kr_A.
$$

The next lemma converts Dirichlet-form recurrence into the reduction property
used below.

\begin{lem}\label{lem:excessive-constant}
If $\cE_K$ is recurrent, then every bounded nonnegative function $g$
satisfying
$$
Kg\le g,
$$
is constant $\mu$-almost everywhere.
\end{lem}

\begin{proof}
We first prove that $Kg=g$ almost everywhere. Fix $\varepsilon>0$ and
write $g_\varepsilon=g+\varepsilon$. Let $u\in L^2(X)$ be bounded and
supported on a set of finite measure, with $\cE_K(u,u)<\infty$. Then
\begin{equation}\label{eq:discrete-picone}
\begin{aligned}
0
&\le
\int_X \frac{u^2}{g_\varepsilon}(g-Kg)\,d\mu\\
&=\int_{\supp u} \frac{u^2}{g_\varepsilon}
(g_\varepsilon-Kg_\varepsilon)\,d\mu\\
&=
\frac{1}{2}\iint_{X\times X}
(g_\varepsilon(x)-g_\varepsilon(y))
\left(
\frac{u(x)^2}{g_\varepsilon(x)}
-
\frac{u(y)^2}{g_\varepsilon(y)}
\right)k(x,y)\,d\mu(x)\,d\mu(y)\\
&=
\frac{1}{2}\iint_{X\times X}
\left(
(u(x)-u(y))^2
-\left(
u(x)\sqrt{\frac{g_\varepsilon(y)}{g_\varepsilon(x)}}
-u(y)\sqrt{\frac{g_\varepsilon(x)}{g_\varepsilon(y)}}
\right)^2
\right)k(x,y)\,d\mu(x)\,d\mu(y)\\
&\le
\cE_K(u,u).
\end{aligned}
\end{equation}

Apply \eqref{eq:discrete-picone} with $u=f_j$, where
$\{f_j\}_{j=1}^\infty$ is a recurrent sequence. Fatou's lemma then yields
$$
0\le\int_X \frac{g-Kg}{g_\varepsilon}\,d\mu
\le\varliminf_{j\to\infty}\cE_K(f_j,f_j)=0.
$$

It follows that $Kg=g$ almost everywhere.

It remains to prove constancy. Choose $C<\infty$ such that
$0\le g\le C$, and put
$$
q=C^2-g^2\ge0.
$$

Since $K1=1$, Jensen's inequality gives $K(g^2)\ge(Kg)^2$. Recalling that
$Kg=g$, we obtain
$$
Kq=C^2-K(g^2)\le C^2-(Kg)^2=q, \mu-a.e.
$$

Applying the first part of the proof to the bounded nonnegative
$K$-excessive function $q$, we conclude that $Kq=q, \mu-a.e.$ Consequently,
$$
K(g^2)=g^2, \mu-a.e.
$$

Therefore, for almost every $x$,
\begin{equation}\label{eq:variance-identity}
\begin{aligned}
0&=K(g^2)(x)-(Kg(x))^2\\
&=\int_X k(x,y)(g(x)-g(y))^2\,d\mu(y).
\end{aligned}
\end{equation}

Since $k(x,y)>0$ for all $x,y\in X$, \eqref{eq:variance-identity} implies
that, for $\mu$-almost every $x$, one has $g(y)=g(x)$ for $\mu$-almost every
$y$. Thus $g$ is constant $\mu$-almost everywhere.
\end{proof}

\begin{prop}\label{prop:full-reduction}
If $\cE_K$ is recurrent and $A\in\B$ satisfies $\mu(A)>0$, then
$$
r_A=1\quad\mu\text{-almost everywhere}
$$
and
$$
H_A(x)=1\qquad\text{for every }x\in X.
$$
\end{prop}

\begin{proof}
Since $r_A$ is bounded, nonnegative, and $K$-excessive,
Lemma~\ref{lem:excessive-constant} implies that $r_A$ is constant almost
everywhere. Since $r_A=1$ on $A$ and $\mu(A)>0$, this constant equals one.
Hence $r_A=1$ almost everywhere.

Set $Z_A=\{x\in X:r_A(x)\ne1\}$. Then $\mu(Z_A)=0$, and the absolute
continuity $\nu_x\ll\mu$ implies that $\nu_x(Z_A)=0$ for every $x\in X$.

Consequently, for every $x\in X$,
$$
H_A(x)-1=\int_X (r_A-1)\,d\nu_x=0.
$$
\end{proof}

\subsection{The Jamison--Orey theorem}
We use the following version of the Jamison--Orey theorem;
see~\cite[Lemma~4, p.~45, and Corollary~1, p.~46]{Jamison_1967}.

Before stating the theorem, we specify the notation in its recurrence
hypothesis. Let $(S,\B)$ be a countably generated measurable space. Given a
transition probability $P$ on $(S,\B)$, let $\Omega=S^{\N}$ and equip
$\Omega$ with the product $\sigma$-algebra $\F=\B^{\otimes\N}$. For
$\omega=(\omega_1,\omega_2,\ldots)\in\Omega$, define the coordinate maps
$$
X_n:\Omega\longrightarrow S,
\qquad
X_n(\omega)=\omega_n,
\qquad
n\ge1.
$$

The coordinate map $X_n$ records the state after $n$ transitions. For each
$x\in S$, let $P_x$ denote the probability measure on
$(\Omega,\F)$ determined by
$$
\begin{aligned}
P_x\left(\bigcap_{j=1}^nX_j^{-1}(A_j)\right)
&=\int_{A_1\times\cdots\times A_n}
P(x,dy_1)P(y_1,dy_2)\cdots P(y_{n-1},dy_n),
\end{aligned}
$$
for $A_1,\ldots,A_n\in\B$. In particular,
$$
P^n(x,A)=P_x(X_n^{-1}(A)).
$$
The existence and uniqueness of $P_x$ follow from the
Ionescu--Tulcea theorem.

\begin{thm}[Jamison--Orey]\label{thm:jamison-orey}
Assume that there is a nonzero $\sigma$-finite measure $m$ such that
$$
P_x\left(\bigcup_{j=1}^{\infty}X_j^{-1}(A)\right)=1,
$$
for every $x\in S$ and every $A\in\B$ with $m(A)>0$.

Then there exists a $\sigma$-finite stationary measure $Q$, unique up to
multiplication by a positive constant. There also exist a positive integer
$d$ and a finite measurable partition
$$
S=C_0\sqcup\cdots\sqcup C_{d-1}\sqcup F,
\qquad
Q(F)=0,
$$
such that
$$
P(x,C_{i+1})=1,
\qquad
x\in C_i,
$$
where the indices are taken modulo $d$. Moreover, if $x,x'\in C_i$ for
some $i$, then
$$
\lim_{n\to\infty}\sup_{A\in\B}
|P^n(x,A)-P^n(x',A)|=0.
$$
\end{thm}

To apply Theorem~\ref{thm:jamison-orey}, let $k_n$ denote the integral
kernel of $K^n$. Explicitly,
$$
\begin{aligned}
k_1&=k,\\
k_{n+1}(x,y)
&=\int_X k(x,z)k_n(z,y)\,d\mu(z)
=\int_X k_n(z,y)\,d\nu_x(z),
\qquad
\forall x,y\in X.
\end{aligned}
$$

\begin{thm}[Row merging]\label{thm:row-merging}
If $\cE_K$ is recurrent, then there exist a positive integer $d$ and a
finite measurable partition
$$
X=C_0\sqcup\cdots\sqcup C_{d-1}\sqcup F,
\qquad
\mu(F)=0,
$$
such that
$$
\nu_x(C_{i+1})=1,
\qquad
x\in C_i,
$$
where the indices are taken modulo $d$. In fact, $d=1$, and for every
$x,x'\in X$,
$$
\lim_{n\to\infty}
\int_X |k_n(x,y)-k_n(x',y)|\,d\mu(y)=0.
$$
\end{thm}

\begin{proof}
Set
$$
P(x,A)\coloneq\nu_x(A)=\int_A k(x,y)\,d\mu(y).
$$

For each $x\in X$, $P(x,\cdot)$ is a probability measure; for each $A\in\B$,
the map $x\mapsto P(x,A)$ is
measurable. Hence $P$ is a transition probability on $(X,\B)$.
The defining iterated-integral formula for $P_x$ and
Tonelli's theorem yield
$$
\begin{aligned}
P^{n+1}(x,A)
&=\int_X P^n(z,A)P(x,dz)\\
&=\int_X\int_A k_n(z,y)k(x,z)\,d\mu(y)\,d\mu(z)\\
&=\int_A k_{n+1}(x,y)\,d\mu(y).
\end{aligned}
$$

It follows by induction from
$P(x,A)=\int_A k_1(x,y)\,d\mu(y)$ that
\begin{equation}\label{eq:transition-density}
P^n(x,A)=K^n1_A(x)=\int_A k_n(x,y)\,d\mu(y),
\end{equation}
for every $n\ge1$.

By \eqref{eq:transition-density}, the full variation of the difference
between two rows is
\begin{equation}\label{eq:total-variation-density}
\int_X |k_n(x,y)-k_n(x',y)|\,d\mu(y)
=2\sup_{A\in\B}|P^n(x,A)-P^n(x',A)|.
\end{equation}

We next verify the recurrence hypothesis of
Theorem~\ref{thm:jamison-orey}. Fix $A\in\B$ with $\mu(A)>0$. For
$m\ge1$, set
$$
E_{A,m}\coloneq\bigcup_{j=1}^mX_j^{-1}(A)\subset\Omega.
$$

We claim that
\begin{equation}\label{eq:hitting-reduction}
P_x(E_{A,m})=H_{A,m}(x),
\end{equation}
for every $m\ge1$ and every $x\in X$. For $m=1$,
\eqref{eq:hitting-reduction} follows
directly from
$$
P_x(E_{A,1})=P(x,A)=K1_A(x)=H_{A,1}(x).
$$

Assume that \eqref{eq:hitting-reduction} holds for some $m\ge1$. Using the
iterated-integral definition of $P_x$ and splitting according to whether the
first coordinate lies in $A$ or $A^c$, we obtain
$$
\begin{aligned}
P_x(E_{A,m+1})
&=\int_X \left(1_A(z)+1_{A^c}(z)H_{A,m}(z)\right)P(x,dz)\\
&=K1_A(x)+K(1_{A^c}H_{A,m})(x)\\
&=H_{A,m+1}(x).
\end{aligned}
$$

Thus \eqref{eq:hitting-reduction} holds for every $m\ge1$ by induction. Since
$E_{A,m}$ increases to
$\bigcup_{j=1}^{\infty}X_j^{-1}(A)$, continuity from below and
Proposition~\ref{prop:full-reduction} give
$$
\begin{aligned}
P_x\left(\bigcup_{j=1}^{\infty}X_j^{-1}(A)\right)
&=\lim_{m\to\infty}P_x(E_{A,m})\\
&=\lim_{m\to\infty}H_{A,m}(x)=H_A(x)=1.
\end{aligned}
$$

Thus the recurrence hypothesis of Theorem~\ref{thm:jamison-orey} holds with
$\mu$ as the recurrence measure.

The measure $\mu$ is stationary for $P$. Indeed,
\eqref{eq:measure-invariance}, applied to $1_A$, gives, for every $A\in\B$,
$$
\int_X P(x,A)\,d\mu(x)=\int_X K1_A(x)\,d\mu(x)=\mu(A).
$$

Since the stationary measure in Theorem~\ref{thm:jamison-orey} is unique up
to a positive scalar multiple, we may rescale it so that $Q=\mu$. The
theorem then yields a finite measurable partition
$$
X=C_0\sqcup\cdots\sqcup C_{d-1}\sqcup F,
\qquad
\mu(F)=0,
$$
such that
$$
\nu_x(C_{i+1})=P(x,C_{i+1})=1,
\qquad
x\in C_i.
$$

We claim that $d=1$. Suppose, to the contrary, that $d\ge2$. Since
$\mu(F)=0$ and $\mu$ is nonzero, at least one cyclic class $C_i$
has positive $\mu$-measure. Choose $x\in C_i$. Since $k$ is strictly
positive,
$$
\nu_x(C_i)=\int_{C_i}k(x,y)\,d\mu(y)>0.
$$

However, $C_i\cap C_{i+1}=\varnothing$ and
$\nu_x(C_{i+1})=1$. Since $\nu_x$ is a probability measure, this
forces $\nu_x(C_i)=0$, a contradiction. Therefore $d=1$, and
$$
X=C_0\sqcup F,
\qquad
\mu(F)=0.
$$

Since $\mu(X\setminus C_0)=0$ and $\nu_x\ll\mu$, we have
$\nu_x(C_0)=1$ for every $x\in X$. Consequently, for $x, x'\in X$,
\begin{equation}\label{eq:row-averaging}
\begin{aligned}
&\int_X |k_{n+1}(x,y)-k_{n+1}(x',y)|\,d\mu(y)\\
&=\int_X \left|
\iint_{C_0\times C_0}
\bigl(k_n(z,y)-k_n(w,y)\bigr)
\,d\nu_x(z)\,d\nu_{x'}(w)
\right|\,d\mu(y)\\
&\le\iint_{C_0\times C_0}
\left(\int_X |k_n(z,y)-k_n(w,y)|\,d\mu(y)\right)
\,d\nu_x(z)\,d\nu_{x'}(w).
\end{aligned}
\end{equation}

For every $z,w\in C_0$, Theorem~\ref{thm:jamison-orey} and
\eqref{eq:total-variation-density} show that the inner integral on the
right-hand side of \eqref{eq:row-averaging} tends to zero. It is bounded by
$2$, since both $k_n(z,\cdot)$ and $k_n(w,\cdot)$ are probability densities.
Since $\nu_x|_{C_0}$ and $\nu_{x'}|_{C_0}$ are probability measures,
dominated convergence applied to \eqref{eq:row-averaging} gives the desired
convergence. Relabelling $n+1$ as $n$ proves the limit in
Theorem~\ref{thm:row-merging}.
\end{proof}

\section{Proofs of the main theorems}

\begin{proof}[Proof of Theorem~\ref{thm:memory-loss}]
Let $m=\dim M$. We apply the framework of
Section~\ref{sec:integral-operators} with
$$
(X,\B,\mu)=\bigl(M,\B(M),\hh_M^m\bigr),
\qquad
k(x,z)=p_1^M(x,z),
\qquad
K=P_1^M.
$$
The results of Section~\ref{sec:preliminaries} show that $\B(M)$ is
countably generated, $\hh_M^m$ is $\sigma$-finite, and $p_1^M$ is strictly
positive and symmetric. Proposition~\ref{prop:parabolic-stochastic-complete}
and \eqref{eq:stochastic-completeness} show that $K1=1$.
Thus all the standing assumptions of
Section~\ref{sec:integral-operators} are satisfied.

Let $\{\eta_j\}$ be the cutoff sequence furnished by
Lemma~\ref{lem:parabolic-cutoff}. Applying the heat-energy estimate
\eqref{eq:heat-energy-estimate} with $t=1$, we obtain
\begin{equation}\label{eq:heat-form-recurrence}
\cE_K(\eta_j,\eta_j)
=\langle\eta_j,(I-P_1^M)\eta_j\rangle_{L^2}
\le\int_M |d\eta_j|^2\,d\hh_M^m
\longrightarrow0.
\end{equation}
Here the convergence follows from \eqref{eq:cutoff-energy-decay}.
Each $\eta_j$ is bounded and compactly supported, so it belongs to $L^2(M)$
and its support has finite $\hh_M^m$-measure. The cutoff sequence converges
pointwise to one. Together with \eqref{eq:heat-form-recurrence}, these facts
show that $\cE_K$ is recurrent in the sense of
Definition~\ref{def:recurrent-form}.

Iterating the semigroup identity in \eqref{eq:heat-kernel-identities} shows
that the kernel $k_n$ of $K^n$ is
$$
k_n(x,z)=p_n^M(x,z).
$$
For $r>0$, set
$$
D_r(x,x')\coloneq
\int_M |p_r^M(x,z)-p_r^M(x',z)|\,d\hh_M^m(z).
$$
Theorem~\ref{thm:row-merging} therefore yields, for all $x,x'\in M$,
\begin{equation}\label{eq:integer-memory-loss}
D_n(x,x')=
\int_M |p_n^M(x,z)-p_n^M(x',z)|\,d\hh_M^m(z)
\longrightarrow0.
\end{equation}

It remains to pass from integer times to arbitrary real times. For $t\ge1$,
write
$$
t=n+s,
\qquad
n=\lfloor t\rfloor,
\qquad
0\le s<1,
$$
and set
$$
f_n(z)\coloneq p_n^M(x,z)-p_n^M(x',z).
$$
For $s>0$, \eqref{eq:heat-kernel-identities} gives
\begin{equation}\label{eq:semigroup-interpolation}
p_t^M(x,\cdot)-p_t^M(x',\cdot)=P_s^M f_n.
\end{equation}
When $s=0$, \eqref{eq:semigroup-interpolation} remains valid with $P_0^M=I$.
Hence
Lemma~\ref{lem:heat-l1-contraction} implies
\begin{equation}\label{eq:real-time-contraction}
D_t(x,x')\le\|f_n\|_{L^1(M)}=D_n(x,x').
\end{equation}
Combining \eqref{eq:integer-memory-loss} and
\eqref{eq:real-time-contraction}, and observing that
$n=\lfloor t\rfloor\to\infty$ as $t\to\infty$, proves
\eqref{eq:memory-loss}.
\end{proof}

\begin{proof}[Proof of Theorem~\ref{thm:main}]
Let $m=\dim M$ and $n=\dim N$.
Proposition~\ref{prop:parabolic-stochastic-complete} shows that $M$ is
stochastically complete. Together with the hypothesis on $N$,
Lemma~\ref{lem:product-heat-kernel} then implies that $M\times N$ is
stochastically complete.

Let $u\in\Hinf(M\times N)$. Corollary~\ref{cor:harmonic-fixed-by-heat} gives
the heat-semigroup invariance \eqref{eq:harmonic-heat-invariance} on
$M\times N$. Combining it with \eqref{eq:product-measure} and
\eqref{eq:product-heat-kernel}, we obtain, for every $t>0$,
\begin{equation}\label{eq:product-heat-representation}
u(x,y)=\int_M\int_N
p_t^M(x,z)p_t^N(y,w)u(z,w)
\,d\hh_N^n(w)\,d\hh_M^m(z).
\end{equation}
This integral is absolutely convergent because $u$ is bounded and, by
\eqref{eq:stochastic-completeness}, both heat kernels have total mass one.

Fix $x,x'\in M$ and $y\in N$. Evaluating
\eqref{eq:product-heat-representation} at $(x,y)$ and $(x',y)$, subtracting,
and using stochastic completeness of $N$, we obtain
\begin{equation}\label{eq:product-memory-loss-estimate}
\begin{aligned}
|u(x,y)-u(x',y)|
&\le\|u\|_{L^\infty(M\times N)}
\int_M\int_N
|p_t^M(x,z)-p_t^M(x',z)|p_t^N(y,w)
\,d\hh_N^n(w)\,d\hh_M^m(z)\\
&=\|u\|_{L^\infty(M\times N)}
\int_M |p_t^M(x,z)-p_t^M(x',z)|\,d\hh_M^m(z).
\end{aligned}
\end{equation}
By \eqref{eq:memory-loss}, the integral on the right-hand side of
\eqref{eq:product-memory-loss-estimate} tends to zero as $t\to\infty$.
Hence $u(x,y)=u(x',y)$, and $u$ is independent of the $M$-variable.

Fix $x_0\in M$ and set $v(y)\coloneq u(x_0,y)$. Then
$u=\pi_N^*v$. By \eqref{eq:product-laplacian} and the independence of the
$M$-variable,
$$
0=\Delta_{M\times N}u(x,y)=\Delta_N v(y).
$$
Thus $v\in\Hinf(N)$.

Conversely, if $v\in\Hinf(N)$, then
$$
\Delta_{M\times N}(\pi_N^*v)=\pi_N^*(\Delta_N v)=0,
$$
so $\pi_N^*v\in\Hinf(M\times N)$. The pullback is linear and injective,
and the construction of $v$ proves surjectivity. Finally, since the projection
$\pi_N:M\times N\to N$ is surjective,
$$
\|\pi_N^*v\|_{L^\infty(M\times N)}
=\|v\|_{L^\infty(N)}.
$$
Therefore $\pi_N^*:\Hinf(N)\to\Hinf(M\times N)$ is an isometric linear
isomorphism. If $N$ is also parabolic, then $\Hinf(N)=\R$, and hence
$\Hinf(M\times N)=\R$.
\end{proof}

\section*{Acknowledgments}

The author would like to express his sincere gratitude to Professor Bobo Hua
for suggesting this problem and for his generous guidance, encouragement, and
support throughout the course of this work.

\bibliographystyle{amsalpha-nodash}
\bibliography{reference}
\end{document}